\documentclass[12pt]{amsart}
\usepackage{amsthm}
\usepackage{amssymb}
\usepackage{cite}
\usepackage{longtable}
\usepackage{geometry}
\usepackage{enumerate}
\usepackage{setspace}
\usepackage{tikz-cd}
\usetikzlibrary{matrix, arrows}
\usepackage[unicode=true]
 {hyperref}
\hypersetup{
colorlinks=true,
urlcolor=black,
citecolor=blue,
linkcolor=blue,
}

\allowdisplaybreaks
\newtheorem{thm}{Theorem}[section]
\newtheorem{prop}[thm]{Proposition}
\newtheorem{lem}[thm]{Lemma}
\newtheorem{cor}[thm]{Corollary}

\theoremstyle{definition}
\newtheorem{definition}[thm]{Definition}

\newtheorem*{assumption}{Assumption}
\theoremstyle{remark}
\newtheorem{remark}[thm]{Remark}

\numberwithin{equation}{section}

\newcommand{\bZ}{\mathbb{Z}}
\newcommand{\bN}{\mathbb{N}}
\newcommand{\Perm}{\mathrm{Perm}}
\newcommand{\Hol}{\mathrm{Hol}}
\newcommand{\NHol}{\mathrm{NHol}}
\newcommand{\Aut}{\mathrm{Aut}}
\newcommand{\Map}{\mathrm{Map}}
\newcommand{\Norm}{\mathrm{Norm}}
\newcommand{\ep}{\epsilon}
\newcommand{\mmod}{\hspace{-3.5mm}\mod}
\newcommand{\pmmod}{\hspace{-2.5mm}\pmod}

\begin{document}

\large 

\title[The multiple holomorph of split metacyclic $p$-groups]{The multiple holomorph of split metacyclic $p$-groups}
\author{Cindy (Sin Yi) Tsang}
\address{Department of Mathematics, Ochanomizu University, Tokyo, Japan}
\email{tsang.sin.yi@ocha.ac.jp}\urladdr{http://sites.google.com/site/cindysinyitsang/} 

\date{\today}

\maketitle

\begin{abstract}Given any group $G$, the normalizer $\mathrm{Hol}(G)$ of the subgroup of left translations in the group of all permutations on $G$ is called the holomorph, and the normalizer $\mathrm{NHol}(G)$ of $\mathrm{Hol}(G)$ in turn is called the multiple holomorph. The quotient $T(G) = \mathrm{NHol}(G)/\mathrm{Hol}(G)$ has been computed for various families of groups $G$ in the literature. In this paper, we shall supplement the existing results by considering finite split metacyclic $p$-groups $G $ with $p$ an odd prime. Our work gives new examples of groups $G$ for which $T(G)$ is not a $2$-group.
\end{abstract}

\tableofcontents

\section{Introduction}

Let $G$ be a group and write $\Perm(G)$ for the group of all permutations on $G$. Recall that a subgroup $N$ of $\Perm(G)$ is called \emph{regular} if the map
\[ N\longrightarrow G;\hspace{1em}\eta \mapsto \eta(1_G)\]
is bijective. For example, the images of the left and right regular representations of $G$, respectively defined by
\[ \begin{cases}
\lambda: G\longrightarrow \Perm(G);\hspace{1em}\lambda(\sigma) = (\tau \mapsto \sigma\tau),\\
\rho: G\longrightarrow \Perm(G);\hspace{1em}\rho(\sigma) = (\tau\mapsto \tau\sigma^{-1}),
\end{cases}\]
are regular subgroups of $\Perm(G)$. The \emph{holomorph of $G$} is defined to be
\[ \Hol(G) = \rho(G)\rtimes \Aut(G).\]
Let $\Norm(\cdot)$ denote normalizer in $\Perm(G)$. Then, it is easy to check that
\[ \Norm(\lambda(G)) = \Hol(G) = \Norm(\rho(G)).\]
The \emph{multiple holomorph of $G$} in turn is defined to be
\[ \NHol(G) = \Norm(\Hol(G)).\]
We are interested in the quotient group
\[ T(G) = \NHol(G)/\Hol(G).\]
The study of $T(G)$ was initiated by G. A. Miller in \cite{Miller}. The motivation was that $T(G)$ acts regularly via conjugation on the set of regular subgroups $N$ of $\Perm(G)$ with $N\simeq G$ and $\Norm(N) = \Hol(G)$. A proof of this  simple fact may be found in \cite[Section 1]{Kohl} or \cite[Section 2]{Tsang ASG}, for example. The structure of $T(G)$ was determined for all finite abelian groups $G$ in \cite{Miller} and later for all finitely generated abelian groups $G$ in \cite{Mills}. Notice that these two papers were published in 1908 and 1951. 

\vspace{1mm}

Research on the group $T(G)$ was revitalized by T. Kohl's paper \cite{Kohl} in 2015, and since then the structure of $T(G)$ has been investigated for other families of groups $G$; see \cite{Caranti1, Caranti2, Caranti3, Tsang ASG, Tsang squarefree, Tsang ADM}. Interestingly, in a lot of the known cases $T(G)$ turns out to be a $2$-group, or even an elementary abelian $2$-group. But there are exceptions and the first such example was given by A. Caranti in \cite{Caranti3}. He showed that $T(G)$ is not a $2$-group for certain $p$-groups $G$ of nilpotency class $2$ with $p$ odd. The present author extended this result slightly to $p$-groups $G$ of nilpotency class at most $p-1$ in \cite{Tsang squarefree}. She also gave examples of groups of the form $G = A\rtimes C$, where $A$ is abelian and $C$ is cyclic of order coprime to the exponent of $A$,  such that $T(G)$ is not a $2$-group. We shall say a bit more about how to construct elements of odd order in $T(G)$ later in Section \ref{last section}.


\vspace{1mm}

The group $T(G)$ acts regularly on and so has the same cardinality as the set of regular subgroups $N$ of $\Perm(G)$ with $N\simeq G$ and $\Norm(N)=\Hol(G)$. These $N$ turn out to be precisely the normal regular subgroups $N$ of $\Hol(G)$ with $N\simeq G$ when $G$ is finite. Again a proof of this simple fact may be found in \cite[Section 1]{Kohl} or \cite[Section 2]{Tsang ASG}. It is easy to see that any regular subgroup $N$ of $\Hol(G)$, not necessarily isomorphic to $G$, must be of the shape
\[ N_\Gamma = \{\rho(\sigma)\Gamma(\sigma) : \sigma \in G\}\mbox{ for some }\Gamma\in\Map(G,\Aut(G)).\]
For $N_\Gamma$ to be a subgroup, in which case regularity is guaranteed, of course $\Gamma$ needs to satisfy certain properties. For $N_\Gamma$ to be a normal subgroup, we have the following nice criterion:

\begin{prop}\label{Gamma lem}For any $\Gamma \in \mathrm{Map}(G,\Aut(G))$, the set $N_\Gamma$ above is a normal (regular) subgroup of $\Hol(G)$ if and only if 
\[ \Gamma(\sigma\tau) = \Gamma(\tau)\Gamma(\sigma)\mbox{ and }\Gamma(\varphi(\sigma)) = \varphi\Gamma(\sigma)\varphi^{-1}\]
hold for all $\sigma,\tau\in G$ and $\varphi\in\Aut(G)$. 
\end{prop}
\begin{proof}See \cite[Theorem 5.2]{Caranti2}. Note that $\gamma$ is used there instead of $\Gamma$, and we changed the notation because $\gamma$ is to denote something else in Section \ref{char section}.
\end{proof}

\begin{remark}\label{remark1}Note that $\rho(G)$ are $\lambda(G)$ are both normal regular subgroups of $\Hol(G)$. They correspond to the maps from $G$ to $\Aut(G)$ defined by
\[ \Gamma_\rho(\sigma) = \mathrm{Id}_G \mbox{ and }\Gamma_\lambda(\sigma) = \mathrm{conj}(\sigma^{-1}) \mbox{ for all }\sigma\in G,\]
respectively, where $\mathrm{conj}(\cdot) = \rho(\cdot)\lambda(\cdot)$ denotes the inner automorphisms.
\end{remark}

Letting $\Aut(G)$ act on $G$ canonically and on itself by conjugation, we may restate Proposition \ref{Gamma lem} as follows: the set $N_\Gamma$ is a normal subgroup of $\Hol(G)$ if and only if $\Gamma$ is an $\Aut(G)$-equivariant antihomomorphism. We now deduce from the above discussion that:

\begin{cor}\label{T(G) cor} For any finite group $G$, the order of $T(G)$ is the number of $\Aut(G)$-equivariant antihomomorphisms $\Gamma$ from $G$ to $\Aut(G)$ with $N_\Gamma\simeq G$. \end{cor}

The purpose of this paper is to study $T(G)$, via these $\Aut(G)$-equivariant antihomomorphisms $\Gamma$, when $G$ is a finite split metacyclic $p$-group with $p$ an odd prime. We may assume that $G$ is non-abelian, because we already know from \cite{Miller} that $T(G)$ is trivial for all abelian groups $G$ of odd order.

\begin{assumption}In the rest of this paper, the symbol $G$ denotes a finite non-abelian split metacyclic $p$-group with $p$ an odd prime. Then, by \cite{King} we know that $G$ has a (unique) presentation of the form
\begin{equation}\label{metacyclic G} G = \langle x,y : x^{p^m} = 1,\,  y^{p^n} = 1,\, yxy^{-1} = x^{1+p^{m-r}}\rangle,\end{equation}
where $m\geq2$, $n\geq1$, and $1\leq r\leq \min\{n,m-1\}$. 
\end{assumption}

Our main result is the following:

\begin{thm}\label{main thm}The order of the quotient $T(G)$ is equal to
\[\begin{cases}
2p^{m-r+\min\{r,n-r\}} &\mbox{when $m\leq n$ with $m-r <r$},\\
(p-1)p^{r-1 + \min\{r,n-r\}} & \mbox{when $m\leq n$ with $r\leq m-r$},\\
(p-1)p^{r-1} &\mbox{when }n\leq m-r.
\end{cases}\]
\end{thm}

We are unable to determine the exact order of $T(G)$ when $m-r<n < m$ because the relevant congruence conditions are too complicated in this case; see the end of Section \ref{char section}. Note that our Theorem \ref{main thm} exhibits another family of groups $G$ for which the order of $T(G)$ is not a power of $2$. It also explains why $T(G)$ has order $18$ for the group $G$ in \cite[Example 3.7]{Tsang squarefree}, which is (\ref{metacyclic G}) with $(p,m,n,r) = (3,3,3,2)$; observe that this group $G$ has nilpotency class $p=3$ and so was not covered by work of \cite{Caranti3} or \cite{Tsang squarefree}.

\vspace{1mm}

Here is an outline of this paper. In Sections \ref{char section} and \ref{iso section}, respectively, we first give an arithmetic characterization of the $\Aut(G)$-equivariant antihomomorphisms $\Gamma$ from $G$ to $\Aut(G)$, and determine the isomorphism class of $N_\Gamma$. In Section \ref{count section}, we prove Theorem \ref{main thm} by counting the number of solutions to certain congruence conditions. Finally in Section \ref{last section}, we shall discuss the actual elements lying in $T(G)$ and compare them with work of \cite{Caranti3} and \cite{Tsang ADM}.

\section{Characterization of equivariant antihomomorphisms}\label{char section}

In this section, we shall give an arithmetic characterization of the $\Aut(G)$-equivariant antihomomorphisms $\Gamma$ from $G$ to $\Aut(G)$. 

\vspace{1mm}

For any $z\in \bZ$ and $\ell\in\bN_{\geq 0}$, let us define
\[ S(z,\ell) = 1 + z + \cdots + z^{\ell-1},\]
with the empty sum representing $0$. For any $i,j\in\bZ$, we then have
\[ (x^iy^j)^{\ell} = x^{iS((1+p^{m-r})^j,\ell)} y^{j\ell}.\]
Below, we record two useful facts, both of which require that $p$ is odd.

\begin{lem}\label{simple lemma}Let $z,\ell,s,t$ be non-negative integers with $s<t$.
\begin{enumerate}[(a)]
\item The integer $z$ satisfies $z^{p^s} \equiv 1\pmod{p^t}$ if and only if $z\equiv 1\pmod{p^{t-s}}$.
\item The exact powers of $p$ dividing $\ell$ and $S(z,\ell)$ are equal if $z\equiv1\pmod{p}$.
\end{enumerate}
\end{lem}
\begin{proof}Part (a) is standard and part (b) is \cite[Lemma 2.1]{Tsang ADM}.
\end{proof}

Now, it is necessary to understand the structure of $\Aut(G)$, which has been computed in \cite{Aut}. We note in passing that the automorphism group of a split metacyclic $2$-group  is also known by \cite{Aut2}.

\begin{prop}\label{Aut prop} The automorphism group of $G$ has order
\[ (p-1)p^{m-1}\cdot p^{\min\{m,n\}}\cdot p^{\min\{m-r,n\}}\cdot p^{n-r},\]
and is a product of four cyclic subgroups, namely
\[ \Aut(G) = \langle\beta\rangle\langle\gamma\rangle\langle\alpha\rangle\langle\delta\rangle,\]
where $\alpha,\beta,\gamma,\delta$, respectively, are automorphisms of orders
\[ (p-1)p^{m-1},\, p^{\min\{m,n\}},\, p^{\min\{m-r,n\}},\, p^{n-r},\]
and are explicitly defined as follows:
\begin{enumerate}[$\bullet$]
\item $\alpha(x) = x^u$ and $\alpha(y) =y$, where $u$ generates the units modulo $p^m$;
\item $\beta(x) = x$ and $\beta(y) = x^{p^{\max\{m-n,0\}}}y$;
\item $\gamma(x) = x y^{p^{\max\{n-m+r,0\}}}$ and $\gamma(y)=y$;
\item $\delta(x) = x$ and $\delta(y) = y^{1+p^r}$.
\end{enumerate}
Moreover, we have the following relations:
\begin{align}\label{conjugation}
&\alpha\delta=\delta\alpha,\,\
\alpha \beta\alpha^{-1}  = \beta^u,\,\
\delta\beta \delta^{-1} = \beta^{(1+p^r)^{-1}},\,\ 
\delta \gamma\delta^{-1} = \gamma^{1+p^r},\\\label{conjugation'}
&\alpha\gamma\alpha^{-1} = \alpha^{a_0}\gamma^{u^{-1}} \mbox{ with }\alpha^{a_0}\gamma = \gamma\alpha^{a_0},\end{align}
where $a_0$ is any natural number satisfying the congruence
\[\label{w eqn}u^{a_0} \equiv uS((1+p^{m-r})^{p^{\max\{n-m+r,0\}}},u^{-1})\pmmod{p^m},\]
and the $(\cdot)^{-1}$ in the exponents are to be interpreted modulo $p^{\min\{m,n\}}$.
\end{prop}
\begin{proof}See \cite[Sections 3 and 4]{Aut} and its corrigendum.\end{proof}

Proposition \ref{Aut prop} implies that elements of $\Aut(G)$ may be written as
\[ \beta^b\gamma^c\alpha^a\delta^d, \mbox{ where }(a,b,c,d)\in\bZ\times\bZ\times\bZ\times\bZ\]
is uniquely determined modulo
\[ (p-1)p^{m-1}\bZ \times p^{\min\{m,n\}}\bZ \times p^{\min\{m-r,n\}}\bZ\times p^{n-r}\bZ.\]
Since an antihomomorphism $\Gamma$ from $G$ to $\Aut(G)$ is uniquely determined by the values of $\Gamma(x)$ and $\Gamma(y)$, it is also clear that:

\begin{lem}\label{reduction lem}For any $\psi_x,\psi_y\in \Aut(G)$, the assignments
\[\Gamma(x)=\psi_x\mbox{ and }\Gamma(y)=\psi_y\]
extend to an antihomomorphism $\Gamma$ from $G$ to $\Aut(G)$ if and only if 
\[ \psi_x^{p^m} = \mathrm{Id}_G,\,\  \psi_y^{p^n} = \mathrm{Id}_G,\,\ \psi_y^{-1}\psi_x\psi_y = \psi_x^{1+p^{m-r}}.\]
Moreover, in this case $\Gamma$ is $\Aut(G)$-equivariant exactly when
\begin{equation}\label{8 relations} \Gamma(\varphi(x)) = \varphi\Gamma(x)\varphi^{-1}\mbox{ and }\Gamma(\varphi(y))=\varphi\Gamma(y)\varphi^{-1}\end{equation}
are both satisfied for $\varphi$ ranging over the generators $\alpha,\beta,\gamma,\delta$.
\end{lem}

The problem is then reduced to determining which of the $\psi_x,\psi_y\in\Aut(G)$ satisfy the conditions in Lemma \ref{reduction lem}. But $\psi_x$ and $\psi_y$ are each determined by four parameters, one for each of the generators $\alpha,\beta,\gamma,\delta$, so approaching this directly could lead to some very complicated calculations. To overcome this issue, we shall first significantly narrow down the possibilities for $\psi_x,\psi_y$, and show that they only require one and two parameters, respectively.

\begin{prop}\label{abd prop} Every $\Aut(G)$-equivariant antihomomorphism $\Gamma$ from $G$ to $\Aut(G)$ satisfies the containments
\begin{enumerate}[(a)]
\item $\Gamma(y) \in\langle\alpha^{(p-1)p^{m-r-1}}\rangle\times\langle\delta^{p^{\max\{n-2r,0\}}}\rangle$;
\item $\Gamma(x)\in\langle\beta^{p^{\min\{m,n\}-r}}\rangle$.
\end{enumerate}
\end{prop}
\begin{proof}[Proof of (a)] Write $\Gamma(y) = \beta^b\gamma^c\alpha^a\delta^d$ with $a,b,c,d\in\bN$. Observe that
\[ \Gamma(y)^{1+p^r} = \Gamma(y^{1+p^r}) = \Gamma(\delta(y)) = \delta\Gamma(y)\delta^{-1}.\]
Assuming that $\beta^b = \gamma^c = \mathrm{Id}_G$, by (\ref{conjugation}) this simplifies to
\[ \alpha^{ap^r}\delta^{dp^r} = \mathrm{Id}_G,\mbox{ and so }\begin{cases}
a\equiv0\pmmod{(p-1)p^{m-r-1}},\\ d\equiv0 \pmmod{p^{\max\{n-2r,0\}}}.
\end{cases}\]
It remains to show that indeed $\beta^b = \gamma^c = \mathrm{Id}_G$. We shall use the relation
\[ \Gamma(y) = \Gamma(\alpha(y)) = \alpha\Gamma(y)\alpha^{-1}.\]
Using (\ref{conjugation}) and (\ref{conjugation'}), we may rewrite the above as
\[ \beta^{b(u-1)}\gamma^{c(u^{-1}-1)}\alpha^{ca_0} = \mathrm{Id}_G.\]
Note that $u-1$ and $u^{-1}-1$ are coprime to $p$ by Lemma \ref{simple lemma}(a). Since $\beta$ and $\gamma$ have orders a power of $p$, we see that $\beta^b = \gamma^c = \mathrm{Id}_G$, as desired.
\end{proof}

\begin{proof}[Proof of (b)] Write $\Gamma(x) = \beta^b\gamma^c\alpha^a\delta^d$ with $a,b,c,d\in\bN$. Observe that
\[ \Gamma(x) = \Gamma(\delta(x)) = \delta\Gamma(x)\delta^{-1}.\]
Assuming that $\gamma^c = \alpha^a = \delta^d = \mathrm{Id}_G$, by (\ref{conjugation}) this simplifies to
\[ \beta^{bp^r}= \mathrm{Id}_G,\mbox{ and so }b\equiv0\pmmod{p^{\min\{m,n\}-r}}.\]
Hence, it remains to show that indeed $\gamma^c = \alpha^a = \delta^d = \mathrm{Id}_G$.
\vspace{1mm}

For $m\leq n$, we have $\beta(y) = xy$, whence
\[ \Gamma(y)\Gamma(x) = \Gamma(\beta(y)) = \beta\Gamma(y)\beta^{-1}, \mbox{ and so } \Gamma(x) = \Gamma(y)^{-1}\beta\Gamma(y)\cdot\beta^{-1}.\] 
Since $\Gamma(y)$ lies in $\langle\alpha\rangle\times\langle\delta\rangle$ by (a), from (\ref{conjugation}) we see that $\langle\beta\rangle$ is normalized by $\Gamma(y)$, so then $\Gamma(x)$ is a power of $\beta$. Thus, we have $\gamma^c = \alpha^a = \delta^d = \mathrm{Id}_G$.

\vspace{1mm}

For $m>n$, we shall first use the relation
\[ \Gamma(x) = \Gamma(\beta(x)) = \beta\Gamma(x)\beta^{-1}.\]
Using (\ref{conjugation}), we may rewrite the above as
\[ \gamma^c = \beta\gamma^c\beta^{-u^a(1+p^r)^{-d}}.\]
In particular, we have the equality
\[ \gamma^c(x) = (\beta\gamma^c\beta^{-u^a(1+p^r)^{-d}})(x)  = (\beta\gamma^c)(x).\] 
By comparing the $\langle x\rangle$-components, we see that
\[ x = x^{1+p^{m-n}S(1+p^{m-r},cp^{\max\{n-m+r,0\}})},\mbox{ and so }c\equiv 0\pmmod{p^{\min\{m-r,n\}}}\]
in view of Lemma \ref{simple lemma}(b).  This shows that $\gamma^c = \mathrm{Id}_G$, and so in fact $\Gamma(x)$ lies in the subgroup $\langle\beta\rangle\rtimes (\langle\alpha\rangle\times\langle\delta\rangle)$.  Next, we consider the relations
\[\Gamma(x)^{p^m} = \mathrm{Id}_G\mbox{ and }
\Gamma(x)^u = \Gamma(x^u) = \Gamma(\alpha(x)) = \alpha\Gamma(x)\alpha^{-1}.\]
By projecting them onto $\langle\alpha\rangle\times\langle\delta\rangle$, we obtain
\[ \alpha^{ap^m}\delta^{dp^m} = \mathrm{Id}_G\mbox{ and }\alpha^{a(u-1)}\delta^{d(u-1)} = \mathrm{Id}_G.\]
Again, we know from Lemma \ref{simple lemma}(a) that $u-1$ is coprime to $p$. The second equation then yields $\delta^{d} = \mathrm{Id}_G$ because $\delta$ has order a power of $p$. From these two equalities, respectively, we also see that
\[a \equiv 0\pmmod{p-1} \mbox{ and }a\equiv0\pmmod{p^{m-1}}.\]
It follows that $\alpha^a = \mathrm{Id}_G$ as well, and this proves the claim.
\end{proof}

In view of Proposition \ref{abd prop}, let us introduce some notation. Put
\[ \widetilde{\alpha} = \alpha^{(p-1)p^{m-r-1}},\,\ \widetilde{\beta} = \beta^{p^{\min\{m,n\}-r}},\,\ \widetilde{\delta} = \delta^{p^{\max\{n-2r,0\}}},\]
and observe that their orders in $\Aut(G)$ are given by
\begin{equation}\label{abd orders}
 \lvert\widetilde{\alpha}\rvert = p^r,\,\ \vert\widetilde{\beta}\rvert = p^r, \,\ \lvert\widetilde{\delta}\rvert = p^{\min\{r,n-r\}}.\end{equation}
Further define the integers
\[ \widetilde{u} = u^{(p-1)p^{m-r-1}},\,\ \widetilde{v} = (1+p^r)^{p^{\max\{n-2r,0\}}}.\]
Then, it follows from the definition that
\[ \widetilde{\alpha}(x) = x^{\widetilde{u}},\,\ \widetilde{\beta}(y) = x^{p^{m-r}}y,\,\ \widetilde{\delta}(y) = y^{\widetilde{v}},\]
and by Lemma \ref{simple lemma}(a), we have
 \begin{equation}\label{uv divisible}p^{m-r}\parallel\widetilde{u}-1,\, 
 p^{\max\{r,n-r\}}\parallel\widetilde{v}-1.\end{equation}
For convenience, we also make the following definition.

\begin{definition}For any $a,b,d\in\bZ$, define
\[ \Gamma_{a,b,d}(x) = \widetilde{\beta}^b \mbox{ and }\Gamma_{a,b,d}(y) = \widetilde{\alpha}^a\widetilde{\delta}^d.\]
The triplet $(a,b,d)$ is called \emph{pre-admissible} if $\Gamma_{a,b,d}$ extends to an antihomomorphism from $G$ to $\Aut(G)$, and \emph{admissible} if $\Gamma_{a,b,d}$ is $\Aut(G)$-equivariant in addition.
\end{definition}

\begin{remark}\label{remark2}Recall Remark \ref{remark1}. It is not hard to see that
\[ \Gamma_\rho = \Gamma_{0,0,0} \mbox{ and }\Gamma_{\lambda} = \Gamma_{a_\lambda,1,0},\]
where $a_\lambda\in \bN$ is such that $\widetilde{u}^{a_\lambda} \equiv (1 + p^{m-r})^{-1}\pmod{p^m}$.\end{remark}

The $\Aut(G)$-equivariant antihomomorphisms $\Gamma$ from $G$ to $\Aut(G)$ are thus precisely the maps $\Gamma_{a,b,d}$ for $(a,b,d)$ ranging over all admissible triplets. We shall now characterize admissibility in terms of congruence conditions.

\begin{prop}\label{pre prop}A triplet $(a,b,d)$ is pre-admissible if and only if
\[ b\widetilde{u}^{-a} \equiv b(1+p^{m-r})\pmmod{p^r}.\]
\end{prop}
\begin{proof}By Lemma \ref{reduction lem}, a triplet $(a,b,d)$ is pre-admissible if and only if
\[ (\widetilde{\beta}^{b})^{p^m} = \mathrm{Id}_G,\, (\widetilde{\alpha}^a\widetilde{\delta}^{d})^{p^n} = \mathrm{Id}_G,\,  (\widetilde{\alpha}^a\widetilde{\delta}^d)^{-1}\widetilde{\beta}^{b}(\widetilde{\alpha}^a\widetilde{\delta}^d) = \widetilde{\beta}^{b(1+p^{m-r})}.\]
Note that the first two equalities always hold by (\ref{abd orders}). For the last equality, using (\ref{conjugation}) we may rewrite the left hand side as
\[(\widetilde{\alpha}^a\widetilde{\delta}^d)^{-1}\widetilde{\beta}^{b}(\widetilde{\alpha}^a\widetilde{\delta}^d)  = \widetilde{\delta}^{-d}\widetilde{\alpha}^{-a}\widetilde{\beta}^b \widetilde{\alpha}^a\widetilde{\delta}^{d} = \widetilde{\beta}^{b\widetilde{u}^{-a}\widetilde{v}^d}.\]
Since $\widetilde{\beta}$ has order $p^r$ and $\widetilde{v}\equiv1\pmod{p^r}$, we see that the claim holds.
\end{proof}

To decide whether a pre-admissible triplet $(a,b,d)$ is in fact admissible, we need to check the two equations in (\ref{8 relations}) for $\varphi$ ranging over $\alpha,\beta,\gamma,\delta$. There are eight relations in total, but it turns that out five of them always hold.

\begin{prop}\label{5 relations prop}For any pre-admissible triplet $(a,b,d)$, the antihomomorphism $\Gamma=\Gamma_{a,b,d}$ satisfies the relations
\begin{align*}
\Gamma(\varphi(x)) &= \varphi\Gamma(x)\varphi^{-1}\mbox{ for }\varphi\in\{\alpha,\beta,\delta\},\\
\Gamma(\varphi(y)) &= \varphi\Gamma(y)\varphi^{-1}\mbox{ for }\varphi\in\{\alpha,\delta\}.
\end{align*}
\end{prop}
\begin{proof}Using (\ref{conjugation}), it is straightforward to check that
\begin{align*}
\Gamma(\alpha(x)) & = \Gamma(x)^u = \widetilde{\beta}^{bu} = \alpha\widetilde{\beta}^b\alpha^{-1} = \alpha\Gamma(x)\alpha^{-1},\\
\Gamma(\beta(x)) & = \Gamma(x) = \widetilde{\beta}^{b} = \beta\widetilde{\beta}^{b}\beta^{-1} = \beta\Gamma(x)\beta^{-1},\\
\Gamma(\alpha(y)) & = \Gamma(y) = \widetilde{\alpha}^a\widetilde{\delta}^d = \alpha\widetilde{\alpha}^{a}\widetilde{\delta}^{d}\alpha^{-1} = \alpha\Gamma(y)\alpha^{-1}.
\end{align*}
Together with (\ref{abd orders}), it is also easy to see that
\begin{align*}
\Gamma(\delta(x)) & = \Gamma(x) = \widetilde{\beta}^{b} = \widetilde{\beta}^{b(1+p^r)^{-1}} = \delta\widetilde{\beta}^{b}\delta^{-1}  = \delta\Gamma(x)\delta^{-1},\\
\Gamma(\delta(y)) & = \Gamma(y)^{1+p^r} = \widetilde{\alpha}^{a(1+p^r)}\widetilde{\delta}^{d(1+p^r)} = \widetilde{\alpha}^{a}\widetilde{\delta}^{d} = \delta\widetilde{\alpha}^{a}\widetilde{\delta}^{d}\delta^{-1} = \delta\Gamma(y)\delta^{-1}.
\end{align*}
Thus, indeed the five stated relations are satisfied.
\end{proof}

For the remaining three relations, two of them are fairly easy to deal with, while $\Gamma(\gamma(x)) = \gamma\Gamma(x)\gamma^{-1}$ is complicated in general. But as we shall show, the congruences in Proposition \ref{3 relations prop}(c) below may be simplified when $m\leq n$ or $n\leq m-r$; this is why we restricted to these two cases in Theorem \ref{main thm}. 
\begin{prop}\label{3 relations prop}Let $(a,b,d)$ be a pre-admissible triplet and put $\Gamma = \Gamma_{a,b,d}$. 
\begin{enumerate}[(a)]
\item The relation $\Gamma(\beta(y)) = \beta\Gamma(y)\beta^{-1}$ holds if and only if
\[\widetilde{u}^{-a}\widetilde{v}^{d}\equiv 1+bp^{m-r} \pmmod{p^{\min\{m,n\}}}.\]
\item The relation $\Gamma(\gamma(y)) = \gamma\Gamma(y)\gamma^{-1}$ holds if and only if
\begin{align}\label{eqn1}
a_0(\widetilde{u}^{-a}-1)&\equiv0\pmmod{p^{m-1}},\\\label{eqn2}
\widetilde{u}^{-a}\widetilde{v}^{d}&\equiv 1\pmmod{p^{\min\{m-r,n\}}},
\end{align}
where $a_0$ is defined as in Proposition \ref{Aut prop}.
\item The relation $\Gamma(\gamma(x)) = \gamma\Gamma(x)\gamma^{-1}$ holds if and only if
\begin{align*}
\widetilde{u}^{aq}(1+bp^{m-r}S(1+p^{m-r},q))&\equiv 1\pmmod{p^m},\\
\widetilde{u}^{aq}bp^{m-r}& \equiv S((1+p^{m-r})^q,bp^{m-r})\pmmod{p^m},\\
\widetilde{v}^{dq} &\equiv 1 \pmmod{p^n},
\end{align*}
where we define $q = p^{\max\{n-m+r,0\}}$.
\end{enumerate}
\end{prop}
\begin{proof}[Proof of (a)]The relation $\Gamma(\beta(y)) = \beta\Gamma(y)\beta^{-1}$ is equivalent to
\[\Gamma(x)^{p^{\max\{m-n,0\}}}\beta = \Gamma(y)^{-1} \beta\Gamma(y),\mbox{ that is }\beta^{bp^{m-r}}\beta =  (\widetilde{\alpha}^{a}\widetilde{\delta}^{d})^{-1}\beta(\widetilde{\alpha}^{a}\widetilde{\delta}^{d}).\]
But by (\ref{conjugation}), we may rewrite the right hand side as
\[  (\widetilde{\alpha}^{a}\widetilde{\delta}^{d})^{-1}\beta(\widetilde{\alpha}^{a}\widetilde{\delta}^{d})= \widetilde{\delta}^{-d}\widetilde{\alpha}^{-a}\beta \widetilde{\alpha}^a\widetilde{\delta}^{d} = \beta^{\widetilde{u}^{-a}\widetilde{v}^{d}}, \]
and the claim is now clear.
\end{proof}

\begin{proof}[Proof of (b)] The relation $\Gamma(\gamma(y)) = \gamma\Gamma(y)\gamma^{-1}$ is equivalent to
\[\gamma = \Gamma(y)\gamma\Gamma(y)^{-1},\mbox{ that is }\gamma = (\widetilde{\alpha}^{a}\widetilde{\delta}^{d})\gamma(\widetilde{\alpha}^{a}\widetilde{\delta}^{d})^{-1}.\]
First, using induction, we deduce from (\ref{conjugation'}) that
\[ \alpha^\ell \gamma \alpha^{-\ell} = \alpha^{a_0S(u^{-1},\ell)}\gamma^{u^{-\ell}}
= \alpha^{a_0(u^{-\ell}-1)/(u^{-1}-1)}\gamma^{u^{-\ell}}\mbox{ for all }\ell\in\bN_{\geq0}.\]
Together with (\ref{conjugation}), we then see that
\[ (\widetilde{\alpha}^{a}\widetilde{\delta}^{d})\gamma(\widetilde{\alpha}^{a}\widetilde{\delta}^{d})^{-1} = \widetilde{\alpha}^{a} \gamma^{\widetilde{v}^d} \widetilde{\alpha}^{-a} = \alpha^{a_0\widetilde{v}^{d}(\widetilde{u}^{-a}-1)/(u^{-1}-1)}\gamma^{\widetilde{u}^{-a}\widetilde{v}^d}.\]
Thus, the relation $\Gamma(\gamma(y)) = \gamma\Gamma(y)\gamma^{-1}$ holds exactly when
\[ \alpha^{a_0\widetilde{v}^{d}(\widetilde{u}^{-a}-1)/(u^{-1}-1)} = \mathrm{Id}_G \mbox{ and }\gamma^{\widetilde{u}^{-a}\widetilde{v}^d} = \gamma.\]
Note that $u^{a_0}\equiv1$ (mod $p$) by definition, so $p-1$ divides $a_0$. Since $u^{-1}-1$ and $\widetilde{v}$ are coprime to $p$, we see that the first equation is equivalent to (\ref{eqn1}). The second equation is clearly equivalent to (\ref{eqn2}), whence the claim.
\end{proof}

\begin{proof}[Proof of (c)] The relation $\Gamma(\gamma(x)) = \gamma\Gamma(x)\gamma^{-1}$ is equivalent to
\[ \Gamma(y)^{q}\Gamma(x)\gamma  = \gamma\Gamma(x),\mbox{ that is }\widetilde{\alpha}^{aq}\widetilde{\delta}^{dq} \widetilde{\beta}^b\gamma = \gamma \widetilde{\beta}^b,\]
where $q= p^{\max\{n-m+r,0\}}$.
A direct computation yields
\begin{align*}
(\widetilde{\alpha}^{aq}\widetilde{\delta}^{dq}\widetilde{\beta}^b\gamma)(x)
& = x^{\widetilde{u}^{aq}(1+ bp^{m-r}S(1+p^{m-r},q))}y^{\widetilde{v}^{dq}q},\\
(\widetilde{\alpha}^{aq}\widetilde{\delta}^{dq}\widetilde{\beta}^b\gamma)(y)
& =x^{\widetilde{u}^{aq}bp^{m-r}} y^{\widetilde{v}^{dq}},\\
(\gamma\widetilde{\beta}^b)(x) &= xy^q,\\
(\gamma\widetilde{\beta}^b)(y) &= x^{S((1+p^{m-r})^q,bp^{m-r})}y^{1+bp^{m-r}q},
\end{align*}
and $y^{1+bp^{m-r}q} = y$. The claim now follows by comparing the exponents.
\end{proof}

To summarize, we have shown that a triplet $(a,b,d)$ is admissible exactly when the congruence conditions in Propositions \ref{pre prop} and \ref{3 relations prop} are all satisfied. Notice that (\ref{eqn2}) follows from the condition in Proposition \ref{3 relations prop}(a) and so may be omitted. Let us further simplify the conditions, as follows.

\begin{lem}\label{b lemma}Every admissible triplet $(a,b,d)$ satisfies
\[b\equiv0\mbox{ or }1\pmmod{p^{\max\{2r-m,0\}}}.\]
\end{lem}
\begin{proof}Recall that $\widetilde{v}\equiv1\pmod{p^r}$. From the conditions in Propositions \ref{pre prop} and \ref{3 relations prop}(a), we then deduce that
\begin{align*}
0 & \equiv b(\widetilde{u}^a(1+p^{m-r})-1)\pmmod{p^r}\\
& \equiv b(\widetilde{u}^a(1+bp^{m-r}) + \widetilde{u}^a(1-b)p^{m-r} - 1)\pmmod{p^r}\\
& \equiv b(\widetilde{v}^d + \widetilde{u}^a(1-b)p^{m-r}-1)\pmmod{p^r}\\
& \equiv \widetilde{u}^ab(1-b)p^{m-r}\pmod{p^r}.
\end{align*}
Since $\widetilde{u}$ is coprime to $p$, this implies that
\[ b(1-b)\equiv0\pmmod{p^{\max\{2r-m,0\}}}.\]
The claim now follows since $b$ and $1-b$ cannot both be divisible by $p$.
\end{proof}

Since $\widetilde{v}\equiv1\pmod{p^r}$, given the two conditions in Proposition \ref{3 relations prop}(a) and Lemma \ref{b lemma}, we may deduce that 
\[ b\widetilde{u}^{-a} \equiv b(1+bp^{m-r}) \equiv b(1+p^{m-r}) \pmmod{p^r}, \]
which is the condition in Proposition \ref{pre prop}. Hence, a triplet $(a,b,d)$ is admissible if and only if the conditions in (\ref{eqn1}), Propositions \ref{3 relations prop}(a),(c), and Lemma \ref{b lemma} are all satisfied. We now specialize to the cases $m\leq n$ and $n\leq m-r$.

\begin{cor}\label{criterion cor}If $m\leq n$, then $(a,b,d)$ is admissible exactly when
\begin{align*}
\widetilde{u}^{-a}\widetilde{v}^d & \equiv 1+ bp^{m-r}\pmmod{p^m},\\
b & \equiv 0\mbox{ or }1 \pmmod{p^{\max\{2r-m,0\}}}.
\end{align*}
\end{cor}
\begin{proof}Suppose that $m\leq n$ and let $q = p^{n-m+r}$. Since $p^r\mid q$, we have
\[ (1+p^{m-r})^q\equiv 1\pmmod{p^m}\]
by Lemma \ref{simple lemma}(a). It then follows from the definition that
\[ u^{a_0}\equiv uS((1+p^{m-r})^q,u^{-1}) \equiv uS(1,u^{-1})\equiv 1\pmmod{p^m},\]
which implies $a_0\equiv 0\pmod{p^{m-1}}$. This means that (\ref{eqn1}) always holds and so may be omitted. Again by Lemma \ref{simple lemma}, and also (\ref{uv divisible}), we have
\[\widetilde{u}^{q} \equiv 1\pmmod{p^m},\,\ \widetilde{v}^{q}\equiv1\pmmod{p^n},\,\ S(1+p^{m-r},q) \equiv0\pmmod{p^r}.\]
We then see that the conditions in Proposition \ref{3 relations prop}(c) hold and thus may be omitted as well. Hence, we are left with the conditions in Proposition \ref{3 relations prop}(a) and Lemma \ref{b lemma}, as claimed.
\end{proof}

\begin{lem}\label{k power}If $r\leq m-r$, then for any $z\in \bZ$ and $\ell\in\bN_{\geq0}$, we have
\[ (1+zp^{m-r})^\ell \equiv 1 + \ell zp^{m-r} \pmmod{p^m},\]
which in particular implies that
\[ S(1+ zp^{m-r},\ell) \equiv \ell + \frac{1}{2}\ell(\ell-1)zp^{m-r}\pmmod{p^m}.\]
\end{lem}
\begin{proof}This follows from the binomial theorem, for example.
\end{proof}

\begin{cor}\label{criterion cor'}If $n\leq m-r$, then $(a,b,d)$ is admissible exactly when
\begin{align*}
\widetilde{u}^{a}(1+bp^{m-r})&\equiv 1\pmmod{p^m},\\
\widetilde{v}^{d} &\equiv 1 \pmmod{p^n}.
\end{align*}
\end{cor}
\begin{proof}Suppose that $n\leq m-r$. We have
\[ u^{a_0} \equiv uS(1+p^{m-r},u^{-1}) \equiv uS(1,u^{-1})\equiv1\pmmod{p^{m-r}}\]
by definition, and this implies that $a_0\equiv0\pmod{p^{m-r-1}}$. 
Since $r\leq m-r$, together with (\ref{uv divisible}), we see that (\ref{eqn1}) always holds and so may be omitted. Note also that the condition in Lemma \ref{b lemma} is vacuous. Now, by plugging in $q=1$, the congruences in Proposition \ref{3 relations prop}(c) become
\begin{align*}
\widetilde{u}^{a}(1+bp^{m-r})&\equiv 1\pmmod{p^m},\\
\widetilde{u}^{a}bp^{m-r}& \equiv S(1+p^{m-r},bp^{m-r})\pmmod{p^m},\\
\widetilde{v}^{d} &\equiv 1 \pmmod{p^n}.
\end{align*}
By Lemma \ref{k power}, the second congruence may be rewritten as
\[ \widetilde{u}^ab \equiv b + \frac{1}{2}b(bp^{m-r}-1)p^{m-r}\equiv b\pmmod{p^r},\]
which follows from the first congruence and so may be omitted. Notice that the condition in Proposition \ref{3 relations prop}(a) may be omitted as well because it follows from the first and third congruences above. Thus, we are only left with the two stated congruences, as claimed. \end{proof}

We have left out the case $m-r < n < m$ here because we do not see any simple way of dealing with the congruences in Proposition \ref{3 relations prop}(c).

\section{Isomorphism classes of normal regular subgroups}\label{iso section}

In Section \ref{char section}, we described the $\Aut(G)$-equivariant antihomomorphisms $\Gamma$ from $G$ to $\Aut(G)$ in terms of suitable congruence conditions. To compute the order of $T(G)$, however, by Corollary \ref{T(G) cor} we only want to count those $\Gamma$ whose associated normal regular subgroup $N_\Gamma$ is isomorphic to $G$.

\vspace{1mm}

In this section, let us fix an admissible triplet $(a,b,d)$, and by definition
\[N_{\Gamma_{a,b,d}} =  \{\rho(x^i)\rho(y^j)\widetilde{\alpha}^{aj}\widetilde{\delta}^{dj}\widetilde{\beta}^{bi} : i,j\in\bZ\}.\]
We shall show that $N_{\Gamma_{a,b,d}}$ is also a split metacyclic $p$-group, isomorphic to a semidirect product of $\bZ/p^m\bZ$ and $\bZ/p^n\bZ$, but it need not be non-abelian. We shall determine the isomorphism class of $N_{\Gamma_{a,b,d}}$ by exhibiting a presentation. As an application, we give a criterion for $N_{\Gamma_{a,b,d}}$ to be isomorphic to $G$.

\vspace{1mm}

Taking $(i,j) = (1,0),(0,1)$, respectively, we obtain the elements
\[ \Phi_x = \rho(x)\widetilde{\beta}^b\mbox{ and }
\Phi_y = \rho(y)\widetilde{\alpha}^a\widetilde{\delta}^d.\]
Note that for any $\ell\in\bN$, we have
\begin{align}\label{powers}
\Phi_x^\ell & = \rho(x\widetilde{\beta}^b(x)\cdots\widetilde{\beta}^{b(\ell-1)}(x))\widetilde{\beta}^{b\ell} = \rho(x^\ell) \widetilde{\beta}^{b\ell},\\\notag
\Phi_y^\ell & = \rho(y(\widetilde{\alpha}^a\widetilde{\delta}^d)(y)\cdots (\widetilde{\alpha}^{a(\ell-1)}\widetilde{\delta}^{d(\ell-1)})(y))\widetilde{\alpha}^{a\ell}\widetilde{\delta}^{d\ell} = \rho(y^{S(\widetilde{v}^d,\ell)})\widetilde{\alpha}^{a\ell}\widetilde{\delta}^{d\ell}.
\end{align}
From this, it is clear that $\langle\Phi_x\rangle$ and $\langle\Phi_y\rangle$ intersect trivially. Since $N_{\Gamma_{a,b,d}}$ has the same order $p^{m+n}$ as $G$, the next lemma shows that $N_{\Gamma_{a,b,d}}$ is also a split metacyclic $p$-group and is given by the semidirect product $\langle\Phi_x\rangle\rtimes \langle\Phi_y\rangle$.

\begin{lem}\label{N order lemma} The elements $\Phi_x$ and $\Phi_y$, respectively, have orders $p^m$ and $p^n$. Moreover, they satisfy the relation $\Phi_y\Phi_x\Phi_y^{-1} = \Phi_x^{\widetilde{u}^{a}(1+ (1-b)p^{m-r}))}$.\end{lem}
\begin{proof}The first claim follows from (\ref{powers}), (\ref{abd orders}), and Lemma \ref{simple lemma}(b). To prove the relation, we compute that
\begin{align*}
\Phi_y\Phi_x\Phi_y^{-1}
& = \rho(y)\widetilde{\alpha}^a\widetilde{\delta}^d \cdot \rho(x)\widetilde{\beta}^b\cdot (\rho(y)\widetilde{\alpha}^a\widetilde{\delta}^d)^{-1}\\
& = \rho(y) \rho((\widetilde{\alpha}^a\widetilde{\delta}^d)(x))\cdot \widetilde{\alpha}^a\widetilde{\delta}^d\widetilde{\beta}^b\widetilde{\delta}^{-d}\widetilde{\alpha}^{-a}\cdot \rho(y^{-1})\\
& = \rho(yx^{\widetilde{u}^a})\widetilde{\beta}^{b\widetilde{u}^a\widetilde{v}^{-d}} \rho(y^{-1}),
\end{align*}
where we used (\ref{conjugation}) in the last equality. But $\widetilde{v}\equiv1\pmod{p^r}$, and we know that $\widetilde{\beta}$ has order $p^r$. We then see that
\begin{align*}
\Phi_y\Phi_x\Phi_y^{-1}& =  \rho(yx^{\widetilde{u}^a})\rho(\widetilde{\beta}^{b\widetilde{u}^a}(y)^{-1})\widetilde{\beta}^{b\widetilde{u}^a} \\
& = \rho(yx^{\widetilde{u}^a}(x^{b\widetilde{u}^ap^{m-r}}y)^{-1}) \widetilde{\beta}^{b\widetilde{u}^a}\\
& = \rho(x^{\widetilde{u}^a(1+p^{m-r})}\cdot x^{-b\widetilde{u}^ap^{m-r}})\widetilde{\beta}^{b\widetilde{u}^a}\\
& = \rho(x^{\widetilde{u}^a(1 + (1-b)p^{m-r})})\widetilde{\beta}^{b\widetilde{u}^a}.
\end{align*}
Recall that $bp^{m-r}\equiv\ep p^{m-r}\pmod{p^r}$ with $\ep \in\{0,1\}$ by Lemma \ref{b lemma}, and so
\[ b(1-b)p^{m-r} \equiv \ep(1-\ep)p^{m-r}\equiv 0\pmmod{p^r}.\]
Since $\widetilde{\beta}$ has order $p^r$, we see that indeed
\begin{align*}
 \Phi_y\Phi_x\Phi_y^{-1} &= \rho(x^{\widetilde{u}^{a}(1+ (1-b)p^{m-r}))})\widetilde{\beta}^{b\widetilde{u}^{a}(1+ (1-b)p^{m-r}))} \\
 & = \Phi_x^{\widetilde{u}^{a}(1+ (1-b)p^{m-r}))},\end{align*}
where the second equality follows from (\ref{powers}).
\end{proof}

To determine whether $N_{\Gamma_{a,b,d}}$ is isomorphic to $G$, first note that
\[ p^{m-r+s_{a,b,d}} \parallel \widetilde{u}^{a}(1+ (1-b)p^{m-r}))-1\mbox{ for some }s_{a,b,d}\in\bN_{\geq 0}\]
by (\ref{uv divisible}), and there exists $j_{a,b,d}\in\bN$ coprime to $p$ such that
\begin{equation}\label{jabd}(\widetilde{u}^{a}(1+ (1-b)p^{m-r}))^{j_{a,b,d}} \equiv 1 + p^{m-r+s_{a,b,d}}\pmmod{p^{m}}\end{equation}
by Lemma \ref{simple lemma}(a). Setting $\Phi_y' = \Phi_y^{j_{a,b,d}}$, we then obtain:

\begin{cor}\label{iso cor} The group $N_{\Gamma_{a,b,d}}$ admits the presentation
\[ N_{\Gamma_{a,b,d}} = \langle \Phi_x,\Phi_y': \Phi_x^{p^m} =1,\,  (\Phi_y')^{p^n} = 1,\, (\Phi_y')\Phi_x(\Phi_y')^{-1} = \Phi_x^{1+p^{m-r+s_{a,b,d}}}\rangle,\]
which is isomorphic to $G$ if and only if $s_{a,b,d}=0$, that is
\begin{equation}\label{iso criterion}\widetilde{u}^{a}(1+ (1-b)p^{m-r})) \not\equiv 1\pmmod{p^{m-r+1}}.\end{equation}
\end{cor}
\begin{proof}The first claim is clear from Lemma \ref{N order lemma}, and the second claim follows from \cite{King}, which tells us that the presentation (\ref{metacyclic G}) is unique.\end{proof}

The above presentation of $N_{\Gamma_{a,b,d}}$ may be also be used to explicitly describe the element in $T(G)$ which $N_{\Gamma_{a,b,d}}$ corresponds to when $N_{\Gamma_{a,b,d}}$ is isomorphic to $G$. Note that when $s_{a,b,d}=0$, we have a well-defined isomorphism
\[ \lambda(G) \longrightarrow N_{\Gamma_{a,b,d}};\hspace{1em}\begin{cases}
\lambda(x)\mapsto \Phi_x,\\
\lambda(y)\mapsto \Phi_y^{j_{a,b,d}}.
\end{cases}\]
Then, as shown in the proof of \cite[Lemma 2.1]{Tsang ASG}, this implies that
\begin{equation}\label{pi N} N_{\Gamma_{a,b,d}} = \pi_{a,b,d}\lambda(G)\pi_{a,b,d}^{-1},\end{equation}
where $\pi_{a,b,d}$ is the bijection defined by
\[ \pi_{a,b,d}:G\longrightarrow G;\hspace{1em}\pi_{a,b,d}(x^iy^j) = (\Phi_x^i\Phi_y^{j_{a,b,d}j})(1).\]
In particular, the element in $T(G)$ which $N_{\Gamma_{a,b,d}}$ gives rise to is $\pi_{a,b,d}\Hol(G)$. Note that $\pi_{a,b,d}$ depends on the choice of $j_{a,b,d}$, which is only unique modulo $p^r$. But say $\pi_{a,b,d}'$ is the bijection arising from a different choice $j_{a,b,d}'$. Then, we have  $j_{a,b,d}^{-1}j_{a,b,d}' \equiv1\pmod{p^r}$, so there exists $\varphi\in\Aut(G)$ which is a power of $\delta$ such that $\varphi(x) = x$ and $\varphi(y) = y^{j_{a,b,d}^{-1}j_{a,b,d}'}$. We see that
\[ \pi_{a,b,d}' = \pi_{a,b,d}\circ\varphi \mbox{ and so }\pi_{a,b,d}'\equiv \pi_{a,b,d}\pmmod{\Aut(G)}.\]
Thus, the element $\pi_{a,b,d}\Hol(G)$ in $T(G)$, which is what we care about, does not depend on the choice of $j_{a,b,d}$. Let us end this section by computing the explicit action of $\pi_{a,b,d}$.

\begin{prop}\label{pi prop}If $s_{a,b,d}=0$, then with a fixed choice of $j_{a,b,d}$, we have
\[ \pi_{a,b,d}(x^iy^j)= x^{-i(1+bp^{m-r}S(k,S(\widetilde{v}^d,j_0j)))k^{-S(\widetilde{v}^d,j_0j)}}y^{-S(\widetilde{v}^d,j_0j)},\]
where we define $k = 1 + p^{m-r}$ and $j_0 = j_{a,b,d}$.
\end{prop}
\begin{proof}From (\ref{powers}), we see that
\begin{align*}
\pi_{a,b,d}(x^iy^j) & = (\rho(x^i)\widetilde{\beta}^{bi} \rho(y^{S(\widetilde{v}^d,j_0j)})\widetilde{\alpha}^{aj_0j}\widetilde{\delta}^{dj_0j})(1)\\
& = (\rho(x^i)\widetilde{\beta}^{bi})(y^{-S(\widetilde{v}^d,j_0j)})\\
& = (x^{bp^{m-r}i}y)^{-S(\widetilde{v}^d,j_0j)} x^{-i}\\
& = y^{-S(\widetilde{v}^d,j_0j)} x^{-bp^{m-r}iS(k,S(\widetilde{v}^d,j_0j))} x^{-i},
\end{align*}
which simplifies to the desired expression.
\end{proof}

\begin{remark}Recall Remarks \ref{remark1} and \ref{remark2}. Notice that we may take $j_{0,0,0} = 1$ and $j_{a_\lambda,1,0} = -1 + p^n$. With these choices, we have
\begin{align*} \pi_{0,0,0}(x^iy^j) & = x^{-ik^{-j}}y^{-j} = (x^iy^j)^{-1},\\
\pi_{a_\lambda,1,0}(x^iy^j) & = x^{-i(1+p^{m-r}S(k,(-1+p^n)j))k^{j}}y^{j} =  x^{-i}y^j,
\end{align*}
where $k=1+p^{m-r}$ as in Proposition \ref{pi prop}, and the last equality holds since
\[ 1 + p^{m-r}S(k,(-1+p^n)j) = 1 + (k-1)S(k,(-1+p^n)j) =  k^{(-1 + p^n)j}.\]
It is then easy to check that
\begin{align*}
\rho(x^iy^j)& = \pi_{0,0,0} \lambda(x^iy^j) \pi_{0,0,0}^{-1},\\
 \lambda(x^{-i}y^j)&= \pi_{a_\lambda,1,0}\lambda(x^iy^j)\pi_{a_\lambda,1,0}^{-1} .
\end{align*}
This verifies (\ref{pi N}) in these two special cases.
\end{remark}

\section{Counting residue classes of admissible triplets}\label{count section}

In this section, we shall prove Theorem \ref{main thm}. First, by Proposition \ref{abd prop}, we know that the $\Aut(G)$-equivariant antihomomorphisms $\Gamma$ from $G$ to $\Aut(G)$ are exactly the $\Gamma_{a,b,d}$ for $(a,b,d)$ ranging over all admissible triplets. Clearly, the definition of $\Gamma_{a,b,d}$ is uniquely determined by the class of  $(a,b,d)$ modulo
\[ \mathbb{M} = p^r\bZ\times p^r\bZ\times p^{\min\{r,n-r\}}\bZ\]
because of (\ref{abd orders}). With Corollaries \ref{T(G) cor} and \ref{iso cor}, we then deduce that
\[ |T(G)| = \#\{\mbox{admissible $(a,b,d)$ mod $\mathbb{M}$ such that (\ref{iso criterion}) holds}\}.\]
In the next two subsections, we shall compute this number using Corollaries \ref{criterion cor} and \ref{criterion cor'}, respectively, in the cases that $m\leq n$ and $n\leq m-r$. We shall in fact first count the number of admissible triplets $(a,b,d)$ modulo $\mathbb{M}$, and then explain how imposing the extra restriction (\ref{iso criterion}) affects the argument.

\vspace{1mm}

Let us first make a change of variables. For any $a\in \bZ$, by (\ref{uv divisible}) and Lemma \ref{simple lemma}(a), there exists $\mu_a\in \bZ$ such that
\begin{equation}\label{mu} \widetilde{u}^a \equiv 1 + \mu_ap^{m-r}\pmmod{p^m},\end{equation}
and $\widetilde{u}$ mod $p^m$ has multiplicative order $p^r$. We then see that
\[ \bZ/p^r\bZ \longrightarrow \bZ/p^r\bZ;\hspace{1em}a+p^r\bZ\mapsto \mu_a+p^r\bZ\]
is a well-defined bijection. With this notation, we have
\[ \widetilde{u}^a(1 + (1-b)p^{m-r}) \equiv 1 + (1 + \mu_a - b)p^{m-r}\pmmod{p^{m-r+1}}.\]
This implies that (\ref{iso criterion}) holds if and only if
\begin{equation}\label{iso criterion'}
1 + \mu_a - b\not\equiv 0\pmmod{p},
\end{equation}
which is much easier to work with.

\subsection{The case $m\leq n$} In this subsection, assume that $m\leq n$. Recall from Corollary \ref{criterion cor} that a triplet $(a,b,d)$ is admissible if and only if
\begin{align}\label{cond1}
\widetilde{u}^{-a}\widetilde{v}^d &\equiv 1 +bp^{m-r}\pmmod{p^m}\\\label{cond2}
b& \equiv0\mbox{ or }1\pmmod{p^{\max\{2r-m,0\}}}
\end{align}
are satisfied. Our strategy is to first choose $b$ satisfying (\ref{cond2}), and then pick $a$ such that (\ref{cond1}) has a solution in $d$.

\begin{prop}\label{count prop} The number of admissible triplets modulo $\mathbb{M}$ is equal to
\[\begin{cases}
2p^{m-r + \min\{r,n-r\}} & \mbox{if $m-r < r$},\\
p^{r+\min\{r,n-r\}} & \mbox{if $r\leq m-r$}.
\end{cases}\]
\end{prop}
\begin{proof}Let us first use (\ref{mu}) to rewrite (\ref{cond1}) as
\[ \widetilde{v}^d \equiv 1 + (\mu_a +b)p^{m-r} + \mu_abp^{2(m-r)}\pmmod{p^m}.\]
We then see that (\ref{cond1}) has a solution in $d$ if and only if
\[ (\mu_a + b )p^{m-r} + \mu_abp^{2(m-r)}\equiv 0\pmmod{p^{\min\{m,\max\{r,n-r\}\}}}\]
because $p^{\max\{r,n-r\}}\parallel \widetilde{v}-1$ by (\ref{uv divisible}). The above is equivalent to
\[\mu_a \equiv -b(1+bp^{m-r})^{-1}\pmod{p^{\min\{r,\max\{r,n-r\}-(m-r)\}}}.\]
This means that once $b$ is fixed, we have 
\[\label{a choices}p^{r-\min\{r,\max\{r,n-r\}-(m-r)\}} = p^{\max\{m- \max\{r,n-r\},0\}}\]
choices for $\mu_a$ and thus $a$ modulo $p^r$. Note that $\widetilde{v}$ mod $p^m$ has multiplicative order $p^{\max\{m-\max\{r,n-r\},0\}}$ by Lemma \ref{simple lemma}(a). Hence, once both $b$ and $\mu_a$ have been chosen, we have
\[ p^{\min\{r,n-r\} - \max\{m-\max\{r,n-r\},0\}} = p^{\min\{r,n-m\}}\]
choices for $d$ modulo $p^{\min\{r,n-r\}}\bZ$. Since there are
\[\label{b choices}\begin{cases}
2p^{m-r} &\mbox{if $m-r<r$}\\
p^r & \mbox{if $r\leq m-r$}
\end{cases}\]
choices for $b$ modulo $p^r$ which satisfy (\ref{cond2}), and
\[ p^{\max\{m- \max\{r,n-r\},0\}}\cdot p^{\min\{r,n-m\}} = p^{\min\{r,n-r\}},\]
the total number of admissible triplets $(a,b,d)$ modulo $\mathbb{M}$ is as claimed.\end{proof}

The case $r\leq m-r$ may actually proven via a different and much simpler argument, as follows.

\begin{prop}\label{count prop'} The number of admissible triplets modulo $\mathbb{M}$ is equal to
\[p^{r+\min\{r,n-r\}}\hspace{1em}\mbox{if $r\leq m-r$}.\]
\end{prop}
\begin{proof}Notice that by (\ref{uv divisible}), the left hand side of (\ref{cond1}) is always congruent to $1$ modulo $p^{m-r}$. This means that given any choices of $a$ and $d$, there exists $b$, which is unique modulo $p^r$, for which (\ref{cond1}) holds. If $r\leq m-r$, then (\ref{cond2}) is vacuous, so there is no other restriction on $b$, whence we have
\[p^r\cdot 1\cdot p^{\min\{r,n-r\}}\]
admissible triplets $(a,b,d)$ modulo $\mathbb{M}$, as claimed.
\end{proof}

We now take condition (\ref{iso criterion}) into account. We consider the three cases:
\begin{enumerate}[(1)]
\item $m-r<r$;
\item $r\leq m-r$ and $m<n$;
\item $r\leq m-r$ and $m=n$.
\end{enumerate}
On the one hand, cases (1) and (2) may be treated using the proof of Proposition \ref{count prop}, except that the number of the choices for $b$ modulo $p^r$ might need to be adjusted to make sure that (\ref{iso criterion}) is satisfied. On the other hand, case (3) may similarly be dealt with using the proof of Proposition \ref{count prop'}, except we must pick $a$ and $d$ suitably so that (\ref{iso criterion}) holds.

\begin{proof}[Proof of Theorem \ref{main thm} when $m\leq n$: Cases (1) and (2)] Observe that
\begin{align*}
 m -r +1& \leq r \leq \max\{r,n-r\}\\
 m-r + 1&\leq n-r = \max\{r,n-r\}
 \end{align*}
in cases (1) and (2), respectively. We then see from (\ref{uv divisible}) that
\[ \widetilde{v} \equiv 1\pmmod{p^{m-r+1}}.\]
Hence, the condition (\ref{cond1}) implies
\[ 1 \equiv 1 + (\mu_a + b)p^{m-r}\pmmod{p^{m-r+1}},\mbox{ namely } \mu_a \equiv -b \pmmod{p}.\]
So by (\ref{iso criterion'}), an admissible triplet $(a,b,d)$ satisfies (\ref{iso criterion}) exactly when 
\[ 1 - 2b\not\equiv0\pmmod{p}.\]
In case (1), this is satisfied by every admissible triplet $(a,b,d)$ by (\ref{cond2}), and so we deduce from Proposition \ref{count prop} that
\[ |T(G)| = 2p^{m-r}\cdot p^{\min\{r,n-r\}}.\]
In case (2), the condition (\ref{cond2}) is vacuous, and requiring $1-2b\not\equiv 0\pmod{p}$ means that instead of $p^r$, we only have $(p-1)p^{r-1}$ choices for $b$ modulo $p^r$. The same argument in Proposition \ref{count prop} then gives us
\[ |T(G)| = (p-1)p^{r-1}\cdot p^{\min\{r,n-r\}}.\]
This proves the theorem when $m\leq n$ in cases (1) and (2).
\end{proof}

\begin{proof}[Proof of Theorem \ref{main thm} when $m\leq n$: Case (3)] Let us make a change of variables for $d$ analogous to (\ref{mu}). Note that $\max\{r,n-r\} = m-r$ in this case. For any $d\in \bZ$, by (\ref{uv divisible}) and Lemma \ref{simple lemma}(a), there exists $\nu_d\in\bZ$ such that
\[ \widetilde{v}^d \equiv 1 + \nu_d p^{m-r} \pmmod{p^m},\]
and $\widetilde{v}$ mod $p^m$ has multiplicative order $p^r$. It follows that
\[ \bZ/p^r\bZ \longrightarrow \bZ/p^r\bZ;\hspace{1em}d+p^r\bZ\mapsto \nu_d+p^r\bZ\]
is a well-defined bijection. Observe that then (\ref{cond1}) implies
\[ bp^{m-r} \equiv (\nu_d - \mu_a)p^{m-r}\pmmod{p^{m-r+1}},\mbox{ namely }b\equiv \nu_d-\mu_a\pmmod{p}.\]
So by (\ref{iso criterion'}), an admissible triplet $(a,b,d)$ satisfies (\ref{iso criterion}) if and only if
\[ 1 + 2\mu_a - \nu_d\not\equiv0\pmmod{p}.\]
This means that we cannot pick both $a$ and $d$ arbitrarily anymore. Instead, once we pick $a$,  we only have $(p-1)p^{r-1}$ choices for $d$ modulo $p^{r}$. From the same argument in Proposition \ref{count prop'}, we then see that 
\[ |T(G)| = p^r \cdot 1 \cdot (p-1)p^{r-1}.\]
This proves the theorem when $m\leq n$ in case (3).
\end{proof}

\subsection{The case $n\leq m-r$} In this subsection, assume that $n\leq m-r$. From Corollary \ref{criterion cor'}, we know that a triplet $(a,b,d)$ is admissible if and only if
\begin{align}\label{cond1'}
\widetilde{u}^a(1+bp^{m-r})& \equiv1\pmmod{p^m} \\\label{cond2'}
\widetilde{v}^d&\equiv1\pmmod{p^n}
\end{align}
are satisfied. Hence, we simply have to pick $(a,b)$ and $d$ such that (\ref{cond1'}) and (\ref{cond2'}) are satisfied, respectively.

\begin{prop}\label{count prop''}The number of admissible triplet modulo $\mathbb{M}$ is equal to $p^r$.
\end{prop}
\begin{proof}By (\ref{uv divisible}) and Lemma \ref{simple lemma}(a), we know that 
\[ \lvert\widetilde{u}\mmod{p^m}\rvert = p^r\mbox{ and }\lvert\widetilde{v}\mmod{p^n}\rvert = p^{\min\{r,n-r\}}.\]
The latter implies that there is only one choice, namely the zero element, for $d$ modulo $p^{\min\{r,n-r\}}\bZ$ such that (\ref{cond2'}) holds. The former implies that $b$ may be chosen arbitrarily, and then $a$ modulo $p^r$ is uniquely determined by (\ref{cond1'}). Thus, indeed we have $1\cdot p^{r}\cdot 1$ admissible triplets $(a,b,d)$ modulo $\mathbb{M}$.
\end{proof}

\begin{proof}[Proof of Theorem \ref{main thm} when $n\leq m-r$] The argument is very similar to that on p. 20. The condition (\ref{cond1'}) implies that
\[ 1 + (\mu_a + b)p^{m-r}\equiv1\pmmod{p^{m-r+1}},\mbox{ namely } \mu_a \equiv -b \pmmod{p}.\]
So by (\ref{iso criterion'}), an admissible triplet $(a,b,d)$ satisfies (\ref{iso criterion}) exactly when 
\[ 1 - 2b\not\equiv0\pmmod{p}.\]
This means that instead of $p^r$, we only have $(p-1)p^{r-1}$ choices for $b$ modulo $p^r$. By the same argument in Proposition \ref{count prop''}, we then obtain
\[ |T(G)| = (p-1)p^{r-1}.\]
This proves the theorem when $n\leq m-r$.
\end{proof}

\section{Elements in the multiple holomorph}\label{last section}

In \cite{Caranti3} and \cite{Tsang ADM}, two different methods of constructing elements in the multiple holomorph were given. In this section, let us recall these constructions, and compare them with the $\pi_{a,b,d}$ calculated in Proposition \ref{pi prop}. It shall also be helpful to recall the definition of $\mathbb{M}$ in Section \ref{count section}.

\vspace{1mm}

First, consider a finite $p$-group $P$. For any $\ell\in\bZ$ coprime to $p$, the map
\[ \pi_\ell : P \longrightarrow P; \hspace{1em} \pi_\ell(\sigma) = \sigma^\ell\]
is a bijection. Of course $\pi_\ell$ need not lie in $\NHol(P)$, and so $\pi_\ell \Hol(P)$ might not be an element of $T(P)$ in general. Nonetheless, in \cite[Proposition 3.1]{Caranti3}, it was shown that if $P$ has nilpotency class $2$, then
\[ \{\pi_\ell \Hol(P) : \ell \in \bZ\mbox{ coprime to $p$}\} \simeq (\bZ/p^e\bZ)^\times\]
is a cyclic subgroup of $T(P)$ whose order is given by
\[ (p-1)p^{e-1},\mbox{ where }\exp(P/Z(P)) = p^e\]
is the exponent of $P/Z(P)$ and $Z(P)$ denotes the center of $P$.

\begin{lem}\label{center lemma}We have $Z(G) = \langle x^{p^r},y^{p^r}\rangle$ and so $\exp(G/Z(G)) = p^r$.
\end{lem}
\begin{proof}For any $i,j\in\bZ$, we have
\begin{align*}
(x^iy^j) x (x^iy^j)^{-1} x^{-1} &= x^{(1+p^{m-r})^j-1},\\
(x^iy^j) y (x^iy^j)^{-1} y^{-1} &= x^{-p^{m-r}i}.
\end{align*}
Since $1+p^{m-r}$ mod $p^m$ has order $p^r$ by Lemma \ref{simple lemma}(a), we see that $x^iy^j$ lies in the center of $G$ if and only if $i,j\equiv0\pmod{p^r}$, whence the claims.\end{proof}

\begin{lem}\label{class lemma}If $r\leq m-r$, then $G$ has nilpotency class $2$.
\end{lem}
\begin{proof}For any $i_1,i_2,j_1,j_2\in\bZ$, we have
\[(x^{i_1}y^{j_1})(x^{i_2}y^{j_2})(x^{i_1}y^{j_1})^{-1}(x^{i_2}y^{j_2})^{-1}  = x^{i_1(1-(1+p^{m-r})^{j_2}) - i_2(1 - (1+p^{m-r})^{j_1})}.\]
Clearly the exponent is divisible by $p^{m-r}$. We then see from Lemma \ref{center lemma} that every commutator lies in $Z(G)$ if $r\leq m-r$, and this implies the claim.
\end{proof}

Lemmas \ref{center lemma} and \ref{class lemma}, together with \cite{Caranti3}, then show that the power maps $\pi_\ell$ define a cyclic subgroup of order $(p-1)p^{r-1}$ in $T(G)$ when $r\leq m-r$. Note that by Theorem \ref{main thm}, this means that
\[ T(G) = \{\pi_\ell\Hol(G) : \ell\in\bZ\mbox{ coprime to }p\}\]
when $n\leq m-r$. Let us now show that these $\pi_\ell$ correspond precisely to the admissible triplets $(a,b,0)$, namely those with $d\equiv 0\pmod{p^{\min\{r,n-r\}}}$ (satisfying (\ref{iso criterion}) so that $N_{\Gamma_{a,b,0}}$ is isomorphic to $G$), in the cases $m\leq n$ with $r\leq m-r$, or $n\leq m-r$. 

\begin{prop}If $m\leq n$ with $r\leq m-r$, or $n\leq m-r$, then the number of admissible triplets $(a,b,0)$ modulo $\mathbb{M}$ satisfying (\ref{iso criterion}) is equal to $(p-1)p^{r-1}$, and all such $(a,b,0)$ satisfies
\[ \pi_{a,b,0} \equiv \pi_{-j_{a,b,0}} \pmmod{\Aut(G)},\]
where $j_{a,b,0}$ is defined as in (\ref{jabd}).
\end{prop}
\begin{proof}Suppose that $m\leq n$ with $r\leq m-r$, or $n\leq m-r$. The first claim essentially follows from the proof of Theorem \ref{main thm}: with the restriction (\ref{iso criterion}) we only have $(p-1)p^{r-1}$ choices for $b$ modulo $p^r$, and once $b$ is chosen $a$ modulo $p^r$ is uniquely determined by (\ref{cond1}) or (\ref{cond1'}).

\vspace{1mm}

For any admissible triplet $(a,b,0)$ satisfying (\ref{iso criterion}), let us write $j_{a,b} = j_{a,b,0}$ and $\ell_{a,b} = - j_{a,b,0}$. On the one hand, observe that
\begin{align*}\pi_{\ell_{a,b}}(x^iy^j)
& = x^{iS((1+p^{m-r})^{j},\ell_{a,b})}y^{j\ell_{a,b}}\\
& = x^{iS(1+jp^{m-r},\ell_{a,b})}y^{j\ell_{a,b}}\\
& = x^{i(\ell_{a,b} + \frac{1}{2}\ell_{a,b}(\ell_{a,b}-1)jp^{m-r})}y^{j\ell_{a,b}}\end{align*}
by Lemma \ref{k power}. On the other hand, since $j_{a,b}$ is coprime to $p$, there exists $\varphi$ which is a power of $\alpha$ such that $\varphi(x) = x^{-\ell_{a,b}}$ and $\varphi(y)=y$. By Proposition \ref{pi prop} and Lemma \ref{k power}, together with $r\leq m-r$, we then see that
\begin{align*}
(\varphi\circ\pi_{a,b,0})(x^iy^j) & = \varphi(x^{-i(1+bp^{m-r} S(1+p^{m-r},j_{a,b}j))(1+p^{m-r})^{-j_{a,b}j}}y^{-j_{a,b}j})\\
& = \varphi(x^{-i(1+ j_{a,b}jbp^{m-r})(1+\ell_{a,b}jp^{m-r})} y^{j\ell_{a,b}})\\
& = \varphi(x^{-i(1 + \ell_{a,b}jp^{m-r}(1-b))}y^{j\ell_{a,b}})\\
& = x^{i(\ell_{a,b} + \ell_{a,b}^2(1-b)jp^{m-r})}y^{j\ell_{a,b}}.
\end{align*}
Recall the notation in (\ref{mu}). Then, from (\ref{cond1}) or (\ref{cond1'}), we deduce that
\[ (1+\mu_ap^{m-r})(1+bp^{m-r})\equiv1\pmmod{p^m},\mbox{ namely }\mu_a \equiv -b \pmmod{p^r}.\]
Also, by Lemma \ref{k power} and $r\leq m-r$, we see that
\begin{align*} (\widetilde{u}^a(1+(1-b)p^{m-r}))^{j_{a,b}} 
& \equiv (1+ j_{a,b}\mu_ap^{m-r})(1+j_{a,b}(1-b)p^{m-r})\pmmod{p^m}\\
& \equiv 1 + j_{a,b}(1 +\mu_a -b)p^{m-r}\pmmod{p^m}\\
& \equiv 1 + j_{a,b}(1-2b)p^{m-r}\pmmod{p^m}.
\end{align*}
By the definition (\ref{jabd}), this implies that
\[j_{a,b}(1-2b)p^{m-r}\equiv p^{m-r}\pmmod{p^{m}}.\]
It then follows that
\begin{align*}
2\ell_{a,b}^2(1-b)p^{m-r} & \equiv  \ell_{a,b}^2(1-2b)p^{m-r} + \ell_{a,b}^2p^{m-r}\pmmod{p^m}\\
& \equiv -\ell_{a,b}p^{m-r} + \ell_{a,b}^2p^{m-r}\pmmod{p^m}\\
& \equiv \ell_{a,b}(\ell_{a,b} - 1)p^{m-r}\pmmod{p^m}.
\end{align*}
We have thus shown that $\varphi \circ \pi_{a,b,0} = \pi_{\ell_{a,b}}$, whence the claim.
\end{proof}

Next, consider a group which is a semidirect product $Q = A\rtimes \langle y\rangle$, where $A$ is any group and $y$ is the generator of order $p^n$ in $G$. For any $v\in\bZ$ with $v\equiv 1\pmod{p}$, in \cite{Tsang ADM} the present author showed  that we have a bijection defined by
\[ \pi_{v}' : Q \longrightarrow Q;
\hspace{1em} \pi_v'((\mathfrak{a} ,y^j)) = (\mathfrak{a},y^{S(v,j)}) \mbox{ for all }\mathfrak{a}\in A \mbox{ and }j\in\bN_{\geq 0}.\]
Again $\pi_{v}'$ might not lie in $\NHol(Q)$. But in \cite{Tsang ADM}, it was proven that $\pi_{v}'$ lies in $\NHol(Q)$ and the order of $\pi_{v}'\Hol(Q)$ in $T(Q)$ is a power of $p$, under suitable hypotheses, one of which is that the exponent of $A$ is coprime to $p$. For our group $G$ in (\ref{metacyclic G}), where $A$ is the cyclic group of order $p^m$, this hypothesis is of course never satisfied. Nevertheless, up to the inversion map 
\[ \iota : G \longrightarrow G;\hspace{1em}\iota(x^iy^j) = (x^iy^j)^{-1},\]
some of the bijections $\pi_{a,b,d}$ do arise in this way when $m\leq n$. 

\begin{prop}\label{last prop}If $m\leq n$, then the number of admissible triplets $(a,0,d)$ modulo $\mathbb{M}$  with $\widetilde{v}^d\equiv1\pmod{p^m}$ is equal to $p^{\min\{r,n-m\}}$, and all such $(a,0,d)$ satisfy (\ref{iso criterion}) as well as the congruence
\[ \iota \circ \pi_{a,0,d} \equiv \pi_{\widetilde{v}^d}'\pmmod{\Aut(G)}.\]\end{prop}
\begin{proof}Suppose that $m\leq n$. Note that $\widetilde{v}^d\equiv1\pmod{p^m}$ is equivalent to
\[ d\equiv 0\pmmod{p^{\max\{m-\max\{r,n-r\},0\}}}\]
by (\ref{uv divisible}) and Lemma \ref{simple lemma}(a). Hence, we have
\[ p^{\min\{r,n-r\} - \max\{m-\max\{r,n-r\},0\}} = p^{\min\{r,n-m\}}\]
choices for $d$ modulo $p^{\min\{r,n-r\}}$. A triplet $(a,0,d)$ with $\widetilde{v}^d\equiv1\pmod{p^m}$ is admissible exactly when $\widetilde{u}^a \equiv 1\pmod{p^m}$ by Corollary \ref{criterion cor}. Thus, by (\ref{uv divisible}), there is only one choice for $a$ modulo $p^r$, namely the zero element, and this proves the first claim.

\vspace{1mm}

For any admissible triplet $(a,0,d)$ with $\widetilde{v}^d\equiv1\pmod{p^m}$, we have
\[ \widetilde{u}^a(1 + (1-0)p^{m-r}) \equiv  1(1 + p^{m-r}) \equiv 1 + p^{m-r}\pmmod{p^m}\]
and so (\ref{iso criterion}) always holds. This also implies that we may take $j_{a,0,d}=1$ for the $j_{a,0,d}$ defined in (\ref{jabd}). With this choice, by Proposition \ref{pi prop}, we have
\begin{align*} (\iota\circ\pi_{a,0,d})(x^iy^j) 
& =  (x^{-i(1+p^{m-r})^{-S(\widetilde{v}^d,j)}}y^{-S(\widetilde{v}^d,j)})^{-1}\\
& = y^{S(\widetilde{v}^d,j)}x^{i(1+p^{m-r})^{-{S(\widetilde{v}^d,j)}}}y^{-S(\widetilde{v}^d,j)}y^{S(\widetilde{v}^d,j)}\\
& = x^iy^{S(\widetilde{v}^d,j)}\\
& = \pi_{\widetilde{v}^d}'(x^iy^j).
\end{align*}
This completes the proof.
\end{proof}

\section*{Acknowledgments} This research was supported by the Fundamental Research Funds for the Central Universities (Award No.: 19lpgy247). The author thanks the referee for helpful comments.

\end{document}